\documentclass{article}

\newcommand{\Addresses}{{
		\bigskip
		\footnotesize
		
		\textsc{Department of Mathematics, Technion - Israel Institute of Technology, Haifa, Israel}\par\nopagebreak
		\textit{E-mail address:} \texttt{ofir.gor@technion.ac.il}

        \medskip

        \textsc{Department of Mathematical Sciences, Durham University, Stockton Road, Durham DH1 3LE}\par\nopagebreak
		\textit{E-mail address:} \texttt{mo-dick.wong@durham.ac.uk}
}}

\makeatletter
\newcommand{\subjclass}[2][2020]{%
  \let\@oldtitle\@title%
  \gdef\@title{\@oldtitle\footnotetext{#1 \emph{Mathematics subject classification.} #2}}%
}
\makeatother
\title{A short proof of Helson's conjecture}

\author{Ofir Gorodetsky, Mo Dick Wong}
\date{}

\usepackage[margin=1in]{geometry}
\usepackage{amssymb}
\usepackage{amsfonts}
\usepackage{amsthm}
\usepackage{amsmath}
\usepackage{hyperref}
\usepackage{cleveref}
\usepackage{mathtools}
\usepackage{url}

\theoremstyle{plain}
\newtheorem{thm}{Theorem}[section]

\newtheorem{lem}[thm]{Lemma}

\theoremstyle{remark}
\newtheorem{rem}{Remark}[section]

\newcommand{\PP}{\mathbb{P}}
\newcommand{\RR}{\mathbb{R}}
\newcommand{\NN}{\mathbb{N}}
\newcommand{\CC}{\mathbb{C}}
\newcommand{\ZZ}{\mathbb{Z}}

\newcommand{\EE}{\mathbb{E}}
\newcommand*\diff{\mathop{}\!\mathrm{d}}

\numberwithin{equation}{section}

\begin{document}

\maketitle

\begin{abstract}
Let $\alpha \colon \NN \to S^1$ be the Steinhaus multiplicative function: a completely multiplicative function such that $(\alpha(p))_{p\text{ prime}}$ are i.i.d.~random variables uniformly distributed on the complex unit circle $S^1$. Helson conjectured that $\EE|\sum_{n\le x}\alpha(n)|=o(\sqrt{x})$ as $x \to \infty$, and this was solved in a strong form by Harper. We give a short proof of the conjecture using a result of Saksman and Webb on a random model for the zeta function. 
\end{abstract}
\section{Introduction}
Let $\alpha$ be the Steinhaus random multiplicative function, defined as follows. If $n$ is a positive integer that factorizes as $\prod_{i=1}^{k} p_i^{a_i}$ ($p_1<p_2<\ldots<p_k$ are primes) then $\alpha(n):=\prod_{i=1}^{k} \alpha(p_i)^{a_i}$, where $(\alpha(p))_{p\text{ prime}}$ are i.i.d.~random variables with the uniform distribution on $S^1$, the complex unit circle $\{z\in \CC: |z| =1\}$. It is not hard to see that for any given positive integers $n$ and $m$,
\begin{equation}\label{eq:stein}
\EE \left[\alpha(n) \overline{\alpha}(m)\right] = \begin{cases} 1& \text{if }n=m,\\ 0& \text{otherwise.}\end{cases}
\end{equation}
Let
\[ S_x := \frac{1}{\sqrt{x}} \sum_{n \le x} \alpha(n).\]
By \eqref{eq:stein}, $\EE \left[|S_x|^2 \right] \asymp 1$. In \cite{Helson}, Helson conjectured that $\lim_{x \to \infty}\EE |S_x| =0$. This conjecture was solved in a strong form by Harper \cite{Har2020}. An elegant and simplified variant of Harper's results, in a model case, was established by Soundararajan and Zaman \cite{SZ}\footnote{See \cite{GZ} for generalizations of the bounds in \cite{SZ}, and \cite[Lemma 7.5]{NPS} for a different derivation of some of the bounds in \cite{SZ}.}. In this note we give a short proof of the following result.
\begin{thm}\label{thm:upper}
Fix $\delta \in (0, 1)$. We have $\EE \left[|S_x|^{2q}\right]\ll (\log \log x)^{-q/2}$ uniformly in $x \ge 3$ and $q \in [0, 1-\delta]$.
\end{thm}
Harper's result is stronger than \Cref{thm:upper} in two ways: it is uniform in $q \in [0,1]$ (which requires modifying the upper bound in the statement), and it contains a matching lower bound. However, \Cref{thm:upper} readily implies Helson's conjecture. The proof of \Cref{thm:upper} still follows the broad strategy in \cite{Har2020}, and in fact was anticipated by Harper \cite[p.~11]{Har2020}. To prove \Cref{thm:upper} we combine two inequalities. Define the function
\[ A_y(s) := \prod_{p \le y}(1-\alpha(p)p^{-s})^{-1},\qquad \Re s >0.\]
\begin{lem}\label{lem:upper}
Fix $\delta\in (0,1)$. Uniformly for $y \in [2,\sqrt{x}]$ and $q \in [0, 1-\delta]$ we have, for some absolute $C\ge 0$ and $c>0$,
\[ \EE \left[|S_x|^{2q}\right] \ll \EE\left[\left( \frac{1}{\log y}\int_{\RR} \left|\frac{A_y(1/2+it)}{1/2+it}\right|^2 \diff{t}\right)^{q}\right]  +((\log y)^C e^{-c\log x / \log y})^q. \]
\end{lem}
\begin{lem}\label{lem:upper2}
Fix $\delta\in (0,1)$. Uniformly for $y \ge 3$ and $q \in [0, 1-\delta]$ we have 
\[\EE\left[\left( \frac{1}{\log y}\int_{\RR} \left|\frac{A_y(1/2+it)}{1/2+it}\right|^2 \diff{t}\right)^{q}\right] \ll (\log \log y)^{-q/2}.\]
\end{lem}
Taking $\log y = \log x/(\log \log x)^2$ in \Cref{lem:upper} and \Cref{lem:upper2} gives \Cref{thm:upper}.

\Cref{lem:upper} and its proof should be viewed as simplified versions of \cite[Proposition 1]{Har2020} and its proof. Our simplification was inspired by a lemma of Najnudel, Paquette and Simm in a model case \cite[Lemma 7.5]{NPS}. The same simplification was also used by Harper in the character sum case \cite[p.~13]{harper2023typical}. 

\Cref{lem:upper2} corresponds to Key Propositions 1 and 2 in \cite{Har2020}. Unlike Harper's self-contained proof which builds on branching process techniques (such as the so-called barrier estimates) and Berestycki's thick-point approach to the construction of Gaussian multiplicative chaos (GMC) \cite{Berestycki}, we follows a philosophy similar to that of Saksman and Webb \cite{SaksmanWebb} (cf.~\cite{SW}). The relevance of \cite{SaksmanWebb} to Helson's conjecture was already hinted in \cite[p.~11]{Har2020}; here we complete the necessary arguments and explain how the main coupling result (i.e.~Gaussian approximation to the logarithm of randomised Riemann zeta function; see \Cref{thm:SW-GMC} below) helps reduce Lemma 1.3 to an analogous moment bound for critical GMC. The advantage of this alternative approach is that it allows one to circumvent various technical estimates by leveraging existing results in the literature of GMC such as moment criteria and Kahane's convexity inequality (see \Cref{lem:cGMCmoments} and its proof below), though we pay the price of losing uniformity when $q$ approaches $1$ due to an application of H\"older's inequality\footnote{Gaussian approximation is also featured in Harper's work, but in the form of Berry--Esseen theorem to provide probability estimates of the correct order for barrier events, which ultimately leads to moment estimates of the correct order. On the other hand, the Gaussian approximation in \cite{SaksmanWebb} and \cite{SW} at the level of random fields allows Saksman and Webb to establish distributional convergence of randomised Riemann zeta function (modulus-squared and renormalised) to some random measure absolutely continuous with respect to a critical GMC measure. With this extra ingredient one could improve \Cref{lem:upper2} and conclude the convergence of renormalised $q$-th moments for $q$ bounded away from $1$ (we omit the details here).}.
\section{Proof of \texorpdfstring{\Cref{lem:upper}}{Lemma \ref{lem:upper}}}
We recall two number-theoretic facts. A positive integer is called $y$-smooth (resp.~$y$-rough) if any prime dividing it is at most $y$ (resp.~strictly greater than $y$). 
Let $\Psi(x,y)$ (resp.~$\Phi(x,y)$) be the number of $y$-smooth (resp.~$y$-rough) integers in $[1,x]$. The first fact, due to Rankin \cite[Theorem 5.3.1]{Cojocaru}, is the upper bound (for $x, y \ge 2$)
\begin{equation}\label{eq:smooth}
\Psi(x,y) \ll x (\log y)^A e^{-c\log x /\log y}
\end{equation}
for some absolute $A \ge 0$ and $c>0$.
We reproduce the proof: if $\alpha >0$ then $\Psi(x,y) \le x^{\alpha}\sum_{p\mid n \implies p \le y} n^{-\alpha}=x^{\alpha}\prod_{p \le y}(1-p^{-\alpha})^{-1}$. Take $\alpha = 1-c/\log y$ and note $\sum_{p \le y} -\log(1-p^{-\alpha}) \ll \sum_{p \le y} p^{-1}\ll \log \log y$ by Mertens' theorem \cite[Theorem 2.7]{MV}. The second fact, due to Brun \cite[Theorem 6.2.5]{Cojocaru}, is the upper bound
\begin{equation}\label{eq:rough}
\Phi(x+H,y)-\Phi(x,y) \ll \frac{H}{\min\{\log y,\log H\}}
\end{equation}
for $x,y, H \ge 2$. We turn to the proof of \Cref{lem:upper}, which we establish with $C=A+1$. Given $y \ge 2$ let $\mathcal{F}_y$ be the $\sigma$-algebra generated by $\{\alpha(p): p \le y\}$. 
As long as $n,m$ are both $y$-rough, the identity
\begin{equation}\label{eq:orth}
\EE\left[ \alpha(n) \overline{\alpha}(m)\mid \mathcal{F}_y\right] = \begin{cases} 1& \text{if }n=m,\\ 0& \text{otherwise,}\end{cases}
\end{equation}
still holds despite the conditioning, using the same argument that gives \eqref{eq:stein}. 
Given $y \ge 2$ we define
\[ S_{x,y} := \frac{1}{\sqrt{x}} \sum_{\substack{n \le x\\ n \text{ is }y\text{-smooth}}} \alpha(n).\]
 Since a positive integer can be written uniquely as $m m'$ where $m$ is $y$-rough and $m'$ is $y$-smooth, we have
\begin{equation}\label{eq:Sxassum}
S_x = \sum_{\substack{1 \le m\le x\\ m \text{ is }y\text{-rough}}} \frac{\alpha(m)}{\sqrt{m}} S_{x/m,y}.
\end{equation}
From \eqref{eq:orth} and \eqref{eq:Sxassum},
\begin{equation}\label{eq:orthapp}
\EE\left[ |S_x|^2 \mid \mathcal{F}_y\right] = |S_{x,y}|^2 + \sum_{\substack{y<m\le x\\ m \text{ is }y\text{-rough}}} m^{-1} |S_{x/m,y}|^2.
\end{equation}
From \eqref{eq:stein} and \eqref{eq:smooth},
\[ \EE \left[|S_{x,y}|^{2} \right]
=  x^{-1}\Psi(x,y) \ll (\log y)^A e^{-c\frac{\log x}{\log y}}.\]
We introduce a parameter $T \in [\sqrt{xy},x]$ and write the $m$-sum in \eqref{eq:orthapp} as $T^{(1)}_{x,y}+T^{(2)}_{x,y}$ where
\begin{equation*}
T^{(1)}_{x,y} := \sum_{\substack{y<m\le T \\ m \text{ is }y\text{-rough}}}m^{-1}|S_{x/m,y}|^2,\qquad 
T^{(2)}_{x,y} := \sum_{\substack{T<m\le x\\ m \text{ is }y\text{-rough}}}m^{-1} |S_{x/m,y}|^2.
\end{equation*}
The expectation of $T^{(1)}_{x,y}$ satisfies
\begin{equation}\label{eq:Tsave} \EE \left[T^{(1)}_{x,y}\right] = x^{-1}\sum_{\substack{y<m\le T\\ m \text{ is }y\text{-rough}}} \Psi(x/m,y)\ll  (\log y)^A\sum_{y<m\le T} m^{-1} e^{-c\frac{\log(x/m)}{\log y}} 
\end{equation}
by \eqref{eq:stein} and \eqref{eq:smooth}. The last expression can be bounded and estimated by a geometric sum:
\begin{multline*}
 (\log y)^A\sum_{y<m\le T} m^{-1} e^{-c\frac{\log(x/m)}{\log y}}= (\log y)^A  e^{-c\frac{\log x}{\log y}}\sum_{y<m\le T} m^{-1} e^{c\frac{\log m}{\log y}} \\
 \ll(\log y)^A  e^{-c\frac{\log x}{\log y}}\sum_{k:\, y/e<e^k\le eT} e^{c\frac{k}{\log y}} \ll 
 (\log y)^{A+1}  e^{-c\frac{\log (x/T)}{\log y}}.
 \end{multline*}
We now treat $T^{(2)}_{x,y}$. Observe $S_{t,y}\sqrt{t}=\sum_{n \le t, \, n \text{ is }y\text{-smooth}} \alpha(n)$ is a function of $\lfloor t\rfloor$ only, that is, $S_{t,y}\sqrt{t} = S_{\lfloor t\rfloor, y}\sqrt{\lfloor t\rfloor} \asymp S_{\lfloor t\rfloor,y} \sqrt{t}$. Setting $r:=\lfloor x/m\rfloor$ we may write
\begin{equation}\label{eq:t2upp}
T^{(2)}_{x,y}\ll \sum_{1\le r <x/T}|S_{r,y}|^2 \sum_{\substack{T<m \le x \\ m \text{ is }y\text{-rough} \\ m \in (x/(r+1),x/r]}}m^{-1} \le \sum_{1\le r <x/T} |S_{r,y}|^2 \frac{\Phi(\frac{x}{r},y)-\Phi(\frac{x}{r+1},y)}{x/r} .
\end{equation} 
By the assumption $T\ge \sqrt{xy}$ and \eqref{eq:rough} we can upper bound the right-hand side of \eqref{eq:t2upp} by 
\[ T^{(2)}_{x,y} \ll \frac{1}{\log y}\sum_{1\le r <x/T}\frac{|S_{r,y}|^2}{r} \ll \frac{1}{\log y} \int_{0}^{x/T} |S_{t,y}|^2 \frac{\diff{t}}{t}.\]
Here we used that $\min\{\log y,\log H\}\asymp \log y$ for $H= x/r - x/(r+1)$, when $r<x/T$ and $T\ge \sqrt{xy}$. In summary, 
\[ \EE\left[ |S_x|^2 \mid \mathcal{F}_y\right] \ll \frac{1}{\log y} \int_{0}^{x/T} |S_{t,y}|^2 \frac{\diff{t}}{t} + X\]
where
\[  X:=|S_{x,y}|^2 + T^{(1)}_{x,y} \ge 0, \qquad \EE \left[X\right]\ll (\log y)^{A+1}e^{-c \frac{\log(x/T)}{\log y}}.\]
By H\"older's inequality and subadditivity of the function $a\mapsto a^q$,
\[\EE \left[|S_{x}|^{2q} \mid \mathcal{F}_{y}\right] \le (\EE \left[|S_{x}|^2 \mid \mathcal{F}_{y}\right])^{q} \ll  \bigg( \frac{1}{\log y} \int_{0}^{x/T} |S_{t,y}|^2 \frac{\diff{t}}{t} \bigg)^{q}  + X^{q}.\]
By the law of total expectation and another application of H\"older and subadditivity,
\begin{align*}
\EE \left[|S_x|^{2q} \right]
&\ll \EE\left[\bigg( \frac{1}{\log y} \int_{0}^{x/T} |S_{t,y}|^2 \frac{\diff{t}}{t} \bigg)^{q}\right]  +(\EE \left[X\right])^{q}\\
& \ll \EE\left[\bigg( \frac{1}{\log y} \int_{0}^{\infty} |S_{t,y}|^2 \frac{\diff{t}}{t} \bigg)^{q} \right] + \left[(\log y)^{A+1}e^{-c\frac{\log(x/T)}{\log y}}\right]^q.
\end{align*}
We take $T=x^{3/4}$ and conclude by applying Parseval's theorem in the form \cite[Equation (5.26)]{MV}
\[ 2\pi \int_{0}^{\infty} \bigg|\sum_{n \le t}f(n)\bigg|^2 \frac{\diff{t}}{t^2}=\int_{\RR} \left|\frac{F_f(1/2+it)}{1/2+it}\right|^2 \diff{t}, \]
where $f$ is any arithmetic function with Dirichlet series $F_f(s):=\sum_{n} f(n)/n^s$ whose abscissa of convergence is smaller than $1/2$; we apply it for $f(n)=\alpha(n)\mathbf{1}_{n\text{ is }y\text{-smooth}}$ and $F_f=A_y$.
\begin{rem}\label{rem:deb}
By a classical result of de Bruijn \cite[Equation (1.9)]{debruijn}, $A=0$ is admissible in \eqref{eq:smooth}. Moreover, one saves a factor of $\log y$ in \eqref{eq:Tsave} by using \eqref{eq:rough}. This shows \Cref{lem:upper} holds with $C=0$.
\end{rem}
\section{Proof of \texorpdfstring{\Cref{lem:upper2}}{Lemma \ref{lem:upper2}}}
Our approach to \Cref{lem:upper2} is based on the theory of multiplicative chaos, the connection to which becomes evident if one considers the Taylor series expansion
\begin{equation}\label{eq:def_Ay}
\begin{split}
A_y\left(\sigma+is\right) 
&= \prod_{p \le y}\left[1-\frac{\alpha(p)}{p^{\sigma + is}}\right]^{-1}
 = \exp \Bigg\{ -\sum_{p \le y} \log \left(1 - \frac{\alpha(p)}{p^{\sigma + is}}\right)\Bigg\}\\
& = \exp \Bigg\{ 
\underbrace{\left[\sum_{p \le y} \frac{\alpha(p)}{p^{\sigma + is}}\right]}_{=: \mathcal{G}_{y, 1}(s; \sigma)}
+ \frac{1}{2}\underbrace{\left[\sum_{p \le y} \left(\frac{\alpha(p)}{p^{\sigma + is}}\right)^2\right]}_{=: \mathcal{G}_{y, 2}(s; \sigma)}
+ \underbrace{\left[\sum_{p \le y} \sum_{j \ge 3} \frac{1}{j}\left(\frac{\alpha(p)}{p^{(\sigma + is)j}}\right)^j\right]}_{=: \mathcal{G}_{y, 3}(s; \sigma)}
\Bigg\}
\end{split}
\end{equation}
for $\sigma \ge \frac{1}{2}$. Since $|\mathcal{G}_{y, 3}(s)| \le \sum_{p} \sum_{j \ge 3} p^{-j/2} < \infty$  uniformly in $y \ge 3$, $\sigma \ge \tfrac{1}{2}$ and $s \in \RR$, and exponential integrability can also be established for $\mathcal{G}_{y, 2}$ (see the statement and proof of \Cref{lem:G2-expmom} below for details), we just need to understand the behaviour of $\exp\left(2 \Re G_{y, 1} (s; \sigma)\right) \diff{s}$ as $y \to \infty$. It turns out that moment estimates for such sequence of random measures were already available if $\alpha(p)$ were Gaussian, and thus our task is to translate these results back to the Steinhaus case by Gaussian approximation.

Our analysis here is closely related to the works of Saksman and Webb who showed the convergence of the sequence of measures 
\[    \frac{\sqrt{\log \log y}}{\log y} \exp \left( 2\Re \mathcal{G}_{y, 1}(s; \tfrac{1}{2})\right)\diff{s}
    \qquad \text{and} \qquad 
    \frac{\sqrt{\log \log y}}{\log y}\left|A_y(\tfrac{1}{2} + is)\right|^2 \diff{s}\]
as $y \to \infty$ (see \cite[Theorem 5]{SaksmanWebb} and \cite[Theorem 1.9]{SW}). While $y$-uniform moment estimates were not explicitly established in these works, we note that the necessary ingredients were already contained in their analysis. For concreteness, we shall recall in \Cref{subsec:MCingredient} their coupling result, and explain how that could lead to the proof of \Cref{lem:upper2}.

In the following, we shall denote $\mathcal{G}_{y, j}^\Re(s; \sigma):= \Re \mathcal{G}_{y, j}(s; \sigma)$, and suppress the dependence on $\sigma$ 
whenever there is no risk of confusion. Moreover, all Gaussian fields are assumed to have zero mean unless otherwise specified.

\subsection{Exponential moments of \texorpdfstring{$\mathcal{G}_{y, 2}$}{Gy2}} \label{subsec:expm-G2}
\begin{lem}\label{lem:G2-expmom}
We have
\[    \sup_{n \in \ZZ, y \ge 3, \sigma \ge \frac{1}{2}} \EE\left[ \sup_{s \in [n, n+1]} e^{\lambda \mathcal{G}_{y, 2}^\Re(s; \sigma)} \right] < \infty \qquad \forall \lambda \in \RR. \]
\end{lem}
\begin{proof}
Let us write $\alpha(p) = e^{i \theta_p}$ where $\theta_p \overset{i.i.d.}{\sim} \mathrm{Uniform}([0,2\pi])$. Using the trigonometric identity $\cos x - \cos y = -2 \sin (\frac{x+y}{2}) \sin(\frac{x-y}{2})$, we have
\begin{align*}
\mathcal{G}_{y, 2}^\Re(s) - \mathcal{G}_{y, 2}^\Re(t) 
& = \sum_{p \le y} p^{-2 \sigma} \left[\cos(2\theta_p - 2s\log p) - \cos(2\theta_p - 2t\log p)\right]\\
& = -\sum_{p \le y} \frac{2}{p^{2 \sigma}}\sin\left(2\theta_p - (s+t)\log p\right) \sin\left((t-s)\log p\right)
\overset{d}{=}  \sum_{p \le y} \frac{2}{p^{2 \sigma}}\sin(2\theta_p) \sin\left((s-t)\log p\right).
\end{align*}
Since $|\frac{2}{p^{2 \sigma}}\sin(2\theta_p) \sin\left((s-t)\log p\right)|\le 2p^{-2\sigma}|s-t| \log p$ and $\sum_p p^{-2} \log^2 p \le \int_1^\infty x^{-2} \log^2 x \diff{x} = 2$, we obtain by Hoeffding's inequality \eqref{eq:Hoeffding2} in \Cref{thm:Hoeffding} that
\begin{equation}\label{eq:G3-Hoeffding}
\PP\left(|\mathcal{G}_{y, 2}^\Re(s) - \mathcal{G}_{y, 2}^\Re(t)| \ge u\right)
\le 2 \exp \left(-\frac{u^2}{2\sum_{p \le y}(2p^{-2 \sigma}|s-t|\log p)^2}\right)
\le 2\exp \left(-\frac{u^2}{16|s-t|^2}\right)
\end{equation}
for all $u \ge 0$ and $s, t \in \RR$.
Comparing \eqref{eq:G3-Hoeffding} to \eqref{eq:chaining-assumption},
let $T \subset \RR$ be any bounded interval of length $|T|$ and define $d(s, t) := 2\sqrt{2}|s-t|$. Following the notations in \Cref{thm:chaining},  the cover number with respect to the metric $d$  satisfies $N(T, d, r) \le 1+ \lfloor 2\sqrt{2}|T| /r \rfloor$ for any $r > 0$, and
\[   \gamma_2(T, d) 
    \ll \int_0^\infty \sqrt{\log \left(1+ \lfloor 2\sqrt{2}|T| /r \rfloor\right)}\diff{r}
    = 2\sqrt{2}|T| \int_0^{1} \sqrt{\log \left(1+ \lfloor u^{-1} \rfloor\right)}\diff{u} \ll |T| \]
by Dudley's entropy bound \eqref{eq:Dudley}. Thus it follows from \Cref{thm:chaining} (with $X_t := \mathcal{G}_{y, 2}^\Re(t)$) that
\[   \PP\left(\sup_{s, t \in T}|\mathcal{G}_{y, 2}^\Re(s) - \mathcal{G}_{y, 2}^\Re(t)| \ge u\right) \le 
    C \exp\left(-\frac{u^2}{C|T|^2}\right) \qquad \forall u \ge 0 \]
for some constant $C>0$ uniformly in $|T| > 0$, $y \ge 3$ and $\sigma \ge \frac{1}{2}$, from which we deduce
\begin{align}
\notag
    &\sup_{y \ge 3, \sigma \ge \frac{1}{2},  |T| \le K } \EE\left[ \sup_{s, t \in T} e^{\lambda\left[\mathcal{G}_{y, 2}^\Re(s; \sigma) - \mathcal{G}_{y, 2}^\Re(t; \sigma)\right]}\right] \\
    \label{eq:expm-G2}
    &\le \sup_{y \ge 3, \sigma \ge \frac{1}{2},  |T| \le K } \left[\sum_{j \ge 1} e^{|\lambda| j} \PP\left(\sup_{s, t \in T}|\mathcal{G}_{y, 2}^\Re(s) - \mathcal{G}_{y, 2}^\Re(t)| \in [j-1, j] \right) \right]
    \le C\sum_{j \ge 1} e^{|\lambda| j} e^{-\frac{(j-1)^2}{CK^2}} 
    < \infty.
\end{align}
On the other hand, since
\[   \mathcal{G}_{y, 2}^{\Re}(t) \overset{d}{=} \sum_{p \le y} p^{-2\sigma} \cos(2\theta_p) \qquad \text{and} \qquad
    \sum_{p \le y} p^{-4\sigma} \le  \sum_{p \le y} p^{-2} \le \frac{1}{2}, \]
another application of Hoeffding's inequality (this time using \eqref{eq:Hoeffding1} in \Cref{thm:Hoeffding}) shows that 
\begin{equation}\label{eq:expm-G2-2}
\sup_{y \ge 3, \sigma \ge \frac{1}{2}, t \in \RR} \EE\left[ e^{\lambda \mathcal{G}_{y, 2}^{\Re}(t)}\right]
\le \sup_{y \ge 3, \sigma \ge \frac{1}{2}} \exp \left( \frac{\lambda^2}{2} \sum_{p \le y} p^{-4\sigma}\right)
\le e^{\lambda^2 / 4} \qquad \forall \lambda \in \RR.
\end{equation}
To conclude our proof, note that
\begin{align*}
    &\sup_{n \in \ZZ, y \ge 3, \sigma \ge \frac{1}{2}} \EE\left[ \sup_{s \in [n, n+1]} e^{\lambda \mathcal{G}_{y, 2}^\Re(s; \sigma)} \right]
    =
    \sup_{n \in \ZZ, y \ge 3, \sigma \ge \frac{1}{2}} \sup_{t \in[n, n+1]} \EE\left[ \left(\sup_{s \in [n, n+1]} e^{\lambda \left[\mathcal{G}_{y, 2}^\Re(s; \sigma) - \mathcal{G}_{y, 2}^\Re(t; \sigma)\right]} \right)
    e^{\lambda \mathcal{G}_{y, 2}^\Re(t; \sigma)}
    \right]
    \\
    & \qquad \le \left(\sup_{n \in \ZZ, y \ge 3, \sigma \ge \frac{1}{2}} \sup_{t \in[n, n+1]}\EE\left[\left(\sup_{s \in [n, n+1]} e^{2\lambda \left[\mathcal{G}_{y, 2}^\Re(s; \sigma) - \mathcal{G}_{y, 2}^\Re(t; \sigma)\right]} \right) \right] \right)^{\frac{1}{2}}
    \left(\sup_{y \ge 3, \sigma \ge \frac{1}{2}, t \in \RR} \EE\left[ e^{2\lambda \mathcal{G}_{y, 2}^\Re(t; \sigma)} \right]\right)^{\frac{1}{2}}
\end{align*}
by Cauchy--Schwarz, and the desired result immediately follows from the two estimates \eqref{eq:expm-G2} and \eqref{eq:expm-G2-2}.
\end{proof}

\subsection{Ingredients from multiplicative chaos theory}\label{subsec:MCingredient}
In this subsection we shall always assume $\sigma = \frac{1}{2}$. Recall $\mathcal{G}^{\Re}_{y, 1}(t)=\Re \sum_{p \le y} \frac{\alpha(p)}{p^{1/2 +it}}$. We aim to show that:
\begin{lem}\label{lem:momentgoal}
For any $q \in [0, 1)$, we have
\[    \sup_{y \ge 3} \EE\left[\left(\frac{\sqrt{\log\log y}}{\log y}\int_{0}^{1} \exp\left(2 \mathcal{G}_{y, 1}^\Re(t) \right) \diff{t} \right)^{q} \right] < \infty. \]
\end{lem}
To establish this claim, we now recall a coupling result from \cite{SaksmanWebb,SW}.
\begin{thm}[{cf.~\cite[Theorem 7 and Lemma 17]{SaksmanWebb} and ~\cite[Theorem 1.7]{SW}}]\label{thm:SW-GMC}
On some suitable probability space one can construct i.i.d.~random variables $\left(\alpha(p)\right)_{p}$ with the uniform distribution on $S^1$ and a collection of (real-valued) random fields $\widetilde{\mathcal{G}}_{y, 1}^\Re$ and $E_y$ on $[0, 1]$ such that
\[     \mathcal{G}_{y, 1}^\Re(\cdot) = \widetilde{\mathcal{G}}_{y, 1}^\Re (\cdot) + E_y(\cdot) \]
simultaneously for all $y \ge 3$ almost surely, where
\begin{itemize}
\item $E_y$ is a sequence of continuous fields which converges uniformly almost surely as $y \to \infty$ and satisfies
\[   \EE\left[\sup_{y \ge 3}\sup_{t \in [0, 1]} e^{\lambda E_y(t)}\right] < \infty \qquad \forall \lambda \in \RR; \]
\item $\widetilde{\mathcal{G}}_{y, 1}^\Re$ is a sequence of continuous Gaussian fields with the property that
\begin{equation}\label{eq:G1-cov-approx}
    \sup_{s, t \in [0, 1], \, y \ge 3} \left|\EE\left[\widetilde{\mathcal{G}}_{y, 1}^\Re(s) \widetilde{\mathcal{G}}_{y, 1}^\Re(t) \right] - \frac{1}{2}\log \left(\frac{1}{|s-t|} \wedge \log y \right) \right| <\infty.
\end{equation}
\end{itemize}
\end{thm}
We also recall a fact about existence of moments of critical Gaussian multiplicative chaos.
\begin{lem}\label{lem:cGMCmoments}
Let $G_T(\cdot)$ be a collection of (real-valued) continuous Gaussian fields on $[0, 1]$ with 
\begin{equation}\label{eq:GMC-mom-cond}
    \sup_{x, y \in [0, 1],\, T > 0} \left|\EE\left[G_T(x) G_T(y) \right] - \log \left(\frac{1}{|x-y|}\right) \wedge  T \right| <\infty.
\end{equation}
Then for any $q \in (0, 1)$, we have
\begin{equation}\label{eq:GMC-mom-criterion}
    \sup_{T > 0} \EE\left[\left(\sqrt{T}\int_{0}^{1} e^{\sqrt{2}G_T(x) - \EE[G_T(x)^2]} \diff{x} \right)^{q} \right] < \infty.
\end{equation}
\end{lem}
\begin{proof}[Sketch of proof]
This claim was established as \cite[Corollary 6]{DKRV2014} when $G_T$ is the white-noise decomposition of $\ast$-scale invariant fields (which ultimately follows from earlier results on multiplicative cascades).
In the general case, let us recall a consequence of Kahane's convexity inequality (see e.g., \cite[Lemma 16]{DKRV2014}): if $Y(\cdot)$ and $Z(\cdot)$ are two continuous Gaussian fields satisfying $\EE[Y(s)Y(t)] \le \EE[Z(s)Z(t)]$ for all $s, t \in [0,1]$, then
\begin{equation}\label{eq:Kahane-conv}
    \EE\left[ \left(\int_0^1 e^{Y(x) - \frac{1}{2} \EE[Y(x)^2]}\diff{x}\right)^q\right]
    \ge \EE\left[ \left(\int_0^1 e^{Z(x) - \frac{1}{2} \EE[Z(x)^2]}\diff{x}\right)^q\right] \qquad \forall q \in (0,1).
\end{equation}
Suppose $G_T(\cdot)$ is the white-noise decomposition of some $\ast$-scale invariant field on $[0,1]$, and $\widetilde{G}_T(\cdot)$ is another collection of continuous Gaussian fields satisfying \eqref{eq:GMC-mom-cond}. Since both collections of Gaussian fields satisfy \eqref{eq:GMC-mom-cond}, there necessarily exists some constant $C > 0$ such that
\[   \EE[G_T(s) G_T(t)] \le \EE[\widetilde{G}_T(s) \widetilde{G}_T(t)] + C \qquad \forall s, t \in [0, 1]. \]
Let us define an independent Gaussian random variable $\mathcal{N}$ with mean $0$ and variance $C$, and set $Y(x) := \sqrt{2}G_T(x)$ as well as $Z(x) := \sqrt{2}\left[\widetilde{G}_T(x) + \mathcal{N}\right]$. Then
\[   \EE[Y(s) Y(t)] = 2 \EE[G_T(s) G_T(t)] \le 2 \left\{\EE[\widetilde{G}_T(s) \widetilde{G}_T(t)] + C \right\} = \EE[Z(s) Z(t)] \qquad \forall s, t \in [0,1], \]
and from \eqref{eq:Kahane-conv} we deduce
\begin{align*}
    \EE\left[\left(\int_{0}^{1} e^{\sqrt{2}G_T(x) - \EE[G_T(x)^2]} \diff{x} \right)^{q} \right]
    &\ge \EE\left[\left(\int_{0}^{1} e^{\sqrt{2}\left(\widetilde{G}_T(x)+\mathcal{N}\right) - \EE[\left(\widetilde{G}_T(x)+\mathcal{N}\right)^2]} \diff{x} \right)^{q} \right]\\
    & =\EE\left[\left(e^{\sqrt{2}q \mathcal{N} - q \EE[\mathcal{N}^2]}\right)\right]\EE\left[\left(\int_{0}^{1} e^{\sqrt{2}\widetilde{G}_T(x) - \EE[\widetilde{G}_T(x)^2]} \diff{x} \right)^{q} \right]
\end{align*}
where the equality follows from independence. Using $\EE\left[e^{\sqrt{2}q \mathcal{N}}\right] = e^{q^2 \EE[\mathcal{N}^2]}$, we see that
\[    \sup_{T > 0} \EE\left[\left(\sqrt{T}\int_{0}^{1} e^{\sqrt{2}\widetilde{G}_T(x) - \EE[\widetilde{G}_T(x)^2]} \diff{x} \right)^{q} \right]
    \le e^{q(1-q)C}\sup_{T > 0}\EE\left[\left(\sqrt{T}\int_{0}^{1} e^{\sqrt{2}G_T(x) - \EE[G_T(x)^2]} \diff{x} \right)^{q} \right] \]
and the bound \eqref{eq:GMC-mom-criterion} for $G_T$ implies an analogous bound for $\widetilde{G}_T$, as claimed.
\end{proof}

\begin{proof}[Proof of \Cref{lem:momentgoal}]
Let $q < q' < 1$. Using H\"older's inequality, we obtain
\begin{align*}
& \sup_{y \ge 3} \EE\left[\left(\frac{\sqrt{\log\log y}}{\log y}\int_{0}^{1} \exp\left(2 \mathcal{G}_{y, 1}^\Re(t) \right) \diff{t} \right)^{q} \right]\\
& \qquad \le \left(\sup_{y \ge 3} \EE\left[\sup_{t \in [0, 1]} e^{(1 - q/q')^{-1}E_y(t)} \right]^{1-q/q'}\right) \sup_{y \ge 3} \EE\left[\left(\frac{\sqrt{\log\log y}}{\log y}\int_{0}^{1} \exp\left(2 \widetilde{\mathcal{G}}_{y, 1}^\Re(t) \right) \diff{t} \right)^{q'} \right]^{q/q'}\\
& \qquad \ll_{q, q'}  \sup_{y \ge 3} \EE\left[\left(\sqrt{\log\log y}\int_{0}^{1} \exp\left(2 \widetilde{\mathcal{G}}_{y, 1}^\Re(t) - 2 \EE[\widetilde{\mathcal{G}}_{y, 1}^\Re(t)^2]\right) \diff{t} \right)^{q'} \right]^{q/q'}
\end{align*}
by \Cref{thm:SW-GMC}, where in the last inequality we used $2\EE[\widetilde{\mathcal{G}}_{y, 1}^\Re(t)^2] = \log \log y + O(1)$ by \eqref{eq:G1-cov-approx}. The claim now follows from \Cref{lem:cGMCmoments} with $T = \log\log y$ and $G_T(t) := \sqrt{2} \widetilde{\mathcal{G}}_{y, 1}^\Re(t)$.
\end{proof}

\begin{proof}[Proof of \Cref{lem:upper2}]
Let  $\max(1, 2q) < r < r' < 2$. We have
\begin{align}
\notag
\EE\left[\left(\frac{1}{\log y}\int_{\RR} \left| \frac{A_y\left(1/2 + it\right)}{1/2 + it}\right|^2 \diff{t} \right)^{q} \right]
& \le \EE\left[\left(\frac{1}{\log y}\int_{\RR} \left| \frac{A_y\left(1/2 + it\right)}{1/2 + it}\right|^2 \diff{t} \right)^{r/2} \right]^{2q/r}\\
\label{eq:subsplit}
&\le \left\{
\sum_{n \in \ZZ} \frac{8}{(1+n^2)^{r/2}} \EE\left[\left(\frac{1}{\log y}\int_{n}^{n+1} \left| A_y\left(1/2 + it\right)\right|^2 \diff{t} \right)^{r/2} \right]
\right\}^{2q/r}
\end{align}
by H\"older's inequality and then the subadditivity of $a \mapsto a^{r/2}$. Since the law of $\alpha(p)$ is rotationally invariant, we have $\left(A_y\left(1/2 + i(t+n)\right), t \in [0, 1]\right) \overset{d}{=} \left(A_y\left(1/2 + it\right), t \in [0, 1]\right)$. In particular, \eqref{eq:subsplit} is equal to
\begin{equation}\label{eq:subsplit2}
     \left(\sum_{n \in \ZZ} \frac{8}{(1+n^2)^{r/2}} \right)^{2q/r} \EE\left[\left(\frac{1}{\log y}\int_{0}^{1} \left| A_y\left(1/2 + it\right)\right|^2 \diff{t} \right)^{r/2} \right]^{2q/r}.
\end{equation}
Using \eqref{eq:def_Ay}, the absolute (deterministic) bound for $|\mathcal{G}_{y, 3}(\cdot)|$ as well as H\"older's inequality, we see that
\begin{align*}
& \EE\left[\left(\frac{1}{\log y}\int_{0}^{1} \left| A_y\left(1/2 + it\right)\right|^2 \diff{t} \right)^{r/2} \right]^{2/r}\\
& \qquad \le \EE\left[\left(\sup_{s \in [0,1]} e^{2G_{y, 3}^{\Re}(s)}\right)^{r/2}\left(\sup_{s \in [0,1]} e^{G_{y, 2}^{\Re}(s)}\right)^{r/2}\left(\frac{1}{\log y}\int_{0}^{1} \exp\left(2 \mathcal{G}_{y, 1}^\Re(t) \right) \diff{t} \right)^{r/2} \right]^{2/r}\\
&\qquad \ll \left(\sup_{y\ge 3} \EE\left[\sup_{s \in [0, 1]} e^{ \frac{r}{2}(1 - \frac{r}{r'})^{-1} \mathcal{G}_{y,2}^\Re(s)}\right]^{\frac{2}{r}- \frac{2}{r'}} \right)
\left(\EE\left[\left(\frac{1}{\log y}\int_{0}^{1} \exp\left(2 \mathcal{G}_{y, 1}^\Re(t) \right) \diff{t} \right)^{\frac{r'}{2}} \right]^{\frac{2}{r'}}\right).
\end{align*}
The first factor on the right-hand side is $\ll_{r, r'} 1$ by \Cref{lem:G2-expmom}, and whereas the second factor is $\ll (\log \log y)^{-1/2}$ by \Cref{lem:momentgoal} with the implicit constant being uniform for $r'$ bounded away from $2$ (as a consequence of H\"older's inequality). Substituting this back to \eqref{eq:subsplit2}, we conclude the proof with an upper bound of order $(\log \log y)^{-q/2}$ and the desired uniformity in $q \in [0, 1-\delta]$.
\end{proof}
\section*{Acknowledgments}
O.G.~has received funding from the European Research Council (ERC) under the European Union's Horizon 2020 research and innovation programme (grant agreement no.~851318). We thank Adam Harper and Brad Rodgers for comments on earlier versions. We express gratitude to the referee for a careful reading of the manuscript and helpful suggestions and corrections. 
\appendix
\section{Probability results}
\begin{thm}[Hoeffding {\cite{Hoeffding}}]\label{thm:Hoeffding}
Let $(X_i)_{i \le n}$ be a collection of independent random variables with $\EE[X_i] = 0$ and $|X_i| \le c_i$ for each $i \le n$. Then $S_n := \sum_{i=1}^n X_i$ satisfies
\begin{align}
\label{eq:Hoeffding1}
\EE\left[ e^{\lambda S_n} \right]
&\le \exp\left(\frac{\lambda^2}{2} \sum_{i=1}^n c_i^2\right) \quad \forall \lambda \in \RR\\
\label{eq:Hoeffding2}
\text{and} \quad
\PP\left(|S_n| \ge u\right)  &\le 2 \exp \left(-\frac{u^2}{2\sum_{i=1}^n c_i^2}\right) \quad \forall u \ge 0.
\end{align}
\end{thm}
\begin{thm}[Generic chaining bound, {cf.~\cite[Equation (2.47)]{chaining}}] \label{thm:chaining}
Let $(X_t)_{t \in T}$ be a collection of zero-mean random variables indexed by elements of a metric space $(T, d)$ satisfying
\begin{equation}\label{eq:chaining-assumption}
    \PP\left(|X_s - X_t| \ge u\right) \le 2 \exp\left(-\frac{u^2}{2d(s, t)^2}\right) \qquad \forall s, t \in T, \quad u \ge 0.
\end{equation}
Then there exists some absolute constant $C_1 > 0$ such that
\[ \PP\left(\sup_{s, t \in T} |X_s - X_t| \ge u\right)
\le C_1 \exp\left(-\frac{u^2}{C_1\gamma_2(T, d)^2}\right) \qquad \forall u \ge 0. \]
The special constant $\gamma_2(T, d)$ can be estimated from above by Dudley's entropy bound: there exists some absolute constant $C_2 > 0$ independent of $(T, d)$ such that
\begin{equation}\label{eq:Dudley}
    \gamma_2(T, d) \le C_2 \int_0^\infty  \sqrt{\log N(T, d, r)} \diff{r}
\end{equation}
where $N(T, d, r)$ denotes the smallest number of balls of radius $r$ (with respect to $d$) needed to cover $T$.
\end{thm}

\bibliographystyle{abbrv}
\bibliography{main}

\Addresses

\end{document}